\documentclass[a4paper]{article}

\usepackage[utf8]{inputenc}
\usepackage{lmodern}
\usepackage{amssymb,amsfonts,amsthm,amsmath,bbm,mathrsfs,aliascnt,mathtools}
\usepackage[english]{babel}
\usepackage[colorlinks]{hyperref}
\usepackage{tikz}
\usetikzlibrary{decorations}
\usetikzlibrary{decorations.pathreplacing}
\tikzset{individu/.style={draw,thick}}

\theoremstyle{plain}
\newtheorem{theorem}{Theorem}[section]
\newtheorem{corollary}[theorem]{Corollary}
\newtheorem{lemma}[theorem]{Lemma}
\newtheorem{proposition}[theorem]{Proposition}

\theoremstyle{definition}

\theoremstyle{remark}
\newtheorem{remark}[theorem]{Remark}

\numberwithin{equation}{section}

\usepackage{tikz}

\makeatletter \newcommand \listoftodos{\section*{Todo list} \@starttoc{tdo}}
\newcommand\l@todo[2]
{\par\noindent \textit{#2}, \parbox{10cm}{#1}\par} \makeatother

\newcommand{\N}{\mathbb{N}}
\newcommand{\Z}{\mathbb{Z}}
\newcommand{\R}{\mathbb{R}}

\newcommand{\e}{\mathrm{e}}
\newcommand{\dd}{\mathrm{d}}
\newcommand{\egaldistr}{\,{\overset{(d)}{=}}\,}

\DeclareMathOperator{\E}{\mathbb{E}}

\renewcommand{\P}{\mathbb{P}}

\renewcommand{\epsilon}{\varepsilon}

\title{Counting the zeros of an elephant random walk}
\author{Jean Bertoin\thanks{Institute of Mathematics, University of Zurich, Switzerland.} }
\date{ }

\begin{document}

\maketitle

\begin{abstract} We study how memory  impacts passages at the origin for a so-called elephant random walk in the diffusive regime.
We observe that the number of zeros  always grows asymptotically like the square root of the time, despite the fact that,  depending on the memory parameter,
first return times to $0$ may have a finite expectation or a fat tail with exponent less than $1/2$. We resolve this apparent paradox by recasting the questions in the framework of scaling limits for Markov chains and self-similar Markov processes.
\end{abstract}


\medskip


\section{Introduction}
\label{sec:introduction}
Motivated by the study of the effects of memory on the asymptotic behavior of non-Markovian stochastic processes, Sch\"utz and Trimper \cite{SchTr}  introduced the so-called elephant random walk. The latter is a nearest neighbor process $S=(S(n))_{n\geq 0}$ on the integer lattice $\Z$; its dynamics depend on a parameter $p\in[0,1)$ that specifies the probability of repetition of certain steps. 
Roughly speaking, at each integer time, the elephant remembers one of its previous steps chosen uniformly at random; then it decides, either
with probability $p$ to repeat this step, or with complementary probability $1-p$ to walk in the opposite direction. The elephant is thus more likely  to continue walking in the average direction it has already taken when  $p>1/2$, whereas for $p<1/2$,  it rather tends  to walk back (and for $p=1/2$, the elephant is fully undecided and its path follows that of a simple symmetric random walk). 

Obviously, the dynamics of an elephant random walk mirror those of a two-color random urn \cite[Chapter 6]{Mahmoud}, say where each ball is either $\oplus$ or $\ominus$. At each integer time, a ball is picked uniformly at random in the urn and is then returned together with a new ball  either of the same sign with probability $p$, or of the opposite sign with probability $1-p$. Although our presentation can be entirely developed in the setting of random urns, we shall rather use the elephant framework for a closer connexion with the literature. 

The asymptotic behavior after a proper rescaling of the elephant random walk is well understood, see  \cite{BaurBer, Bercu, ColGavSch1, ColGavSch2,KuTa}\footnote{As it has been pointed out in \cite{BaurBer}, Janson \cite{Janson} proved more general results in the setting of random urns. It has also  been observed by Kubota and Takei \cite{KuTa} that in the case $p>1/2$, the steps of an elephant random walk can be described in terms of a so-called correlated  Bernoulli process, and some of the results above derive directly from an earlier work of Heyde \cite{Heyde}; see also further references quoted in  \cite{KuTa}.
}.
 In short, the regime $p<3/4$ is diffusive, in the sense that 
$(n^{-1/2} S(\lfloor nt\rfloor))_{t\geq 0}$ converges in distribution as $n\to \infty$ to some centered Gaussian process. The regime 
 $p>3/4$ is superdiffusive with $n^{1-2p} S(n) \to   L$ a.s.
for some non-degenerate random variable $L$, and in turn $n^{-1/2}(S(n)- n^{2p-1}L)$ converges to some centered Gaussian variable. Many more results are now known on this process and its variations  which have attracted a growing interest in the recent years, as one can see from a search engine (in particular, note that \cite{Gut} has a title similar to ours, but actually treats a different question).

The elephant random walk can be viewed as a member of the family of reinforced processes; see \cite{Pem} for a survey. The general question of how reinforcement affects the long time behavior of processes  has been intensively investigated for many years, notably in the setting of  edge linear reinforcement for which we shall now recall some of the most significant  results. Angel \textit{et al.} \cite{ACK} proved positive recurrence on any graph with bounded degrees for sufficiently small initial weight. Establishing recurrence on the two-dimension integer lattice $\Z^2$  with arbitrary initial constant weights remained an open problem for many years until the works of Merkl and Rolles \cite{MerRol} and Sabot and Zeng \cite{SabZen}.
In a different direction, Sabot \cite{Sabot} showed that directed edge reinforced random walks are transient on $\Z^d$ for $d\geq 3$.

The analysis of recurrence versus transience for linearly reinforced processes which we briefly mentioned above suggests that one should also consider elephant random walks without rescaling and  investigate
quantitatively the role of the parameter $p$ on the frequency of visits to the origin. We thus consider the counting process of zeros,
\begin{equation}\label{E:defZ}
Z(n)\coloneqq \mathrm{Card}\{1\leq j \leq n: S(j)=0\}, \qquad n\geq 0.
\end{equation}
In terms of a two-color random urn, $Z(n)$ corresponds to the number of times before time $n$ at which the urn contains the same number of balls $\oplus$ as of balls $\ominus$. 
Recently, Coletti and Papageorgiou \cite[Theorems 3.1 and 3.3]{ColPapa} pointed out that $\lim_{n\to \infty} Z(n)=\infty$ a.s.  if $p\leq 3/4$,
whereas $\lim_{n\to \infty} Z(n)<\infty$ a.s. if $p> 3/4$.
In other words,  the elephant random walk is recurrent for $p\leq 3/4$ and transient for $p>3/4$.

In this work, we focus on the diffusive regime, that is we henceforth suppose that $p<3/4$ and \textit{a fortiori} the elephant is recurrent.
We shall first show that 
\begin{equation} \label{E:asympZ}\lim_{n\to \infty} Z(n)/\sqrt n= V \qquad \text{in distribution},
\end{equation}
where $V$ is some non-degenerate random variable whose law depends on $p$. So roughly speaking,   for any parameter $0\leq p <3/4$, the number of zeros  before time $n$  grows as $\sqrt n$ when $n\to \infty$.
 We might  then argue that, in comparison to the simple random walk,  the dynamics with memory of the elephant random walk in the diffusive regime do not significantly  alter 
 the number of visits to the origin (see however the forthcoming Remark \ref{Rema}). Likely, this should not come as much of a surprise, since we are precisely considering a diffusive regime. 
 
 This could be the end of a not so thrilling story. Now recall (see e.g. Feller \cite[Theorem 7]{Feller49}) that for a \textit{time-homogeneous} Markov chain on a discrete state space, the number of  passages to the origin before time $n$ grows like $\sqrt n$ if and only if  the tail distribution of the first return time to the origin
decays like $1/\sqrt n$. The elephant random walk is a \textit{time-inhomogeneous} Markov chain (although some works in the literature improperly assert its non-Markovian character). One might be tempted to think that the same should still hold, that is,  if we write 
$$R\coloneqq \inf\{j\geq 1: S(j)=0\}$$
 for the first return time of $S$ to the origin (in terms of two-color random urns, $R$ is the first time at which the urn contains the same number of balls $\oplus$ as of balls $\ominus$), then one should have  $\P(R>n) \approx 1/\sqrt n$ as $n\to \infty$ for any $0\leq p <3/4$.  However this is far from being the case; indeed  Coletti and Papageorgiou \cite[Theorem 3.2]{ColPapa} proved recently that $\E(R)<\infty$ whenever $p<1/6$.
 We shall 
investigate further the distribution of the first return time, and  in particular, establish the following tail estimate\footnote{This is valid for $p> 0$ only, as for $p=0$, the first return  to the origin occurs at time $2$ a.s.  For $p=0$, if we write $R_2=\inf\{j\geq 3: S(j)=0\}$ for the second return time to zero, then Theorem \ref{T1} yields
$\P(R_2>n)  \sim   2\sqrt{6 n^{-3}/{\pi }}$ as $n\to \infty.$}
\begin{equation}\label{E:T1}
\P(R>n)  \sim   \frac{ 1}{\Gamma(2p)} \sqrt{\frac{6-8p}{\pi }} n^{2p-3/2}\qquad \text{as } n\to \infty.
 \end{equation}
Note that, as a consequence, 
$$\E(R)<\infty \text{ if and only if }p<1/4,$$which improves upon \cite[Theorem 3.2]{ColPapa}.

Roughly speaking, \eqref{E:T1} might suggest that the number of zeros $Z(n)$ should grow approximately linearly in $n$ for $p<1/4$ and like $n^{(3-4p)/2}$ for $p>1/4$,
hence disagreeing with \eqref{E:asympZ}
(note also that the discrepancy disappears for $p=1/2$, which is the case when $S$ is a simple symmetric random walk). 
This reasoning is only informal, and the purpose of this work is not only to establish \eqref{E:asympZ} and \eqref{E:T1}, but also to understand why these two asymptotic behaviors actually agree one with the other.

We will analyze more precisely the tail behavior of the first return time to the origin after time $k$, conditionally on $S(k)=0$, and this uniformly in $k$.
We will show in Theorem \ref{T1} that an estimate similar to \eqref{E:T1} holds, essentially provided that the right-hand side there is multiplied by a factor depending on $k$. 
In short, when $k\approx n$, this factor is of order $n^{1-2p}$ and as a result, the estimate for the tail of the first return time after $k$ is then of order $1/\sqrt n$.
This enables us to apply a general criterion for the scaling limit of Markov chains on $\N$ in \cite{BeKor}, and resolve the apparent paradox in the framework of self-similar Markov processes. 

 Technically, the starting point of our approach is inspired from \cite{ColGavSch2}. Specifically, there is a sequence $(a_n)$ of real factors which turn the elephant random walk into a martingale $(a_n S(n))$, so that  the latter can be embedded in a Brownian path. The fact that $S$ is a nearest-neighbor process entails that $(a_n S(n))$ is a so-called binary splitting martingale, and specificities of the embedding enable us to express the counting process $Z$  of the zeros of the elephant in terms of the excursions of the Brownian motion away from $0$. This is the key to the proof of Theorem \ref{TZ}, which is a  multidimensional and more explicit version of \eqref{E:asympZ}.  Theorem \ref{TZ} can be viewed as a relative of L\'evy's downcrossing theorem for the Brownian local time; and Brownian excursion theory plays a key role for its proof. The proof of Theorem \ref{T1}  requires a much finer analysis of the Brownian embedding, and notably concentration inequalities for which we use the Burkholder-Davis-Gundy inequalities and universal bounds for the moments of sums of independent variables due to Latała \cite{Latala}.

 The plan of this paper is as follows. In Section 2, we present some background on the elephant random walk and its embedding in a Brownian path. 
 We then establish in Section 3 a scaling limit theorem for the counting process of the zeros. We prove uniform estimates for the tail distribution of first return times in Section 4,
 and for this, we establish concentration inequalities for the embedding. Finally, the reconciliation of the two preceding limit theorems is achieved in Section 5, using the framework of scaling limits of Markov chains on $\N$ and self-similar Markov processes.

  Let us also briefly discuss some natural open problems in this area.  It would be very interesting to address the question of recurrence and transience in higher dimension, notably 
  for the multi-dimensional elephant random walk which has been considered recently by Bercu and Laulin \cite{BerLau1, BerLau2}; see also \cite{Marco}. Likely, one should  be able to deduce from local limit theorems as in \cite{FHM}, that in dimension $3$ and above, the expected number of zeros is finite and \textit{a fortiori} the elephant is always transient. In dimension $2$, the expected number of zeros should be infinite in the diffusive regime;  however one cannot conclude that recurrence holds, because of the failure of the  Markov property.  The problem of deciding between transience and recurrence for an elephant random walk in dimension $2$ is thus open to the best of my knowledge. Its solution, not to mention the finer analysis of the asymptotic behavior of return times to the origin, likely requires new ideas. 
 Actually,  I have been surprised to realize that some innocent looking questions for a seemingly simple one-dimensional process like the elephant random walk could be rather delicate to analyze. Considering their two-dimensional versions then seems quite challenging. 

 \textsl{We end this introduction with an important convention that will be used throughout this text. 
On many occasions, we shall need to bound quantities that depend on one or two variables of time, often denoted by $k$ and $n$.
We shall then write  $c$ for some constant which depends neither of $k$ nor of $n$, and  may be different in different expressions.
Possibly, such constant may depend on further parameter(s), say $q$ (typically a power, like in the Burkholder-Davis-Gundy inequality),
in which case we may use the notation $c_q$; of course the same convention as above applies. 
For the sake of notational simplicity, dependence in the memory parameter $p$ is systematically omitted. }

\section{Background on a Brownian embedding}

We start by defining formally the elephant random walk in the diffusive regime.  We recall that the  memory parameter $p\in [0,3/4)$ has been fixed and this will not be mentioned any further.
One considers a random binary sequence $X_1, X_2, \ldots$ in $\{+1,-1\}$  such that
$$X_{n+1} \coloneqq \varepsilon_n X_{u(n)}\qquad \text{for every }n\geq 1,$$
where all the variables $u(1), u(2), \ldots$ and $\varepsilon_1, \varepsilon_2, \ldots$ are independent, 
$$u(n) \text{ has the uniform distribution on }\{1,\ldots, n\}$$
and 
$$\P(\varepsilon_n=1)=p=1-\P(\varepsilon_n=-1).$$
The  elephant random walk (with memory parameter $p$) is the process $S=(S(n))_{n\geq 0}$ of the partial sums
$$S(n)\coloneqq X_1+\cdots + X_n, \qquad n\geq 0.$$

It should be plain that replacing the first step $X_1$ by its opposite $-X_1$ while keeping the dynamics unchanged simply yields the reflected path $-S$. As a consequence, the zeros of $S$  that we are interested in, are independent of  $X_1$. For the sake of simplicity, we shall henceforth also assume that 
$$\P(X_1=1)=\P(X_1=-1)=1/2,$$ 
which thus induces no loss of generality; and the law of $S$ only depends on the memory parameter $p$.

We readily see that the  conditional law of $X_{n+1}$
 given $X_1, \ldots, X_n$  is that of a Bernoulli variable with values $\pm 1$ and  parameter  $1/2+(2p-1) S(n)/2n$. Since this  quantity
 only depends on $n$ and $S(n)$, the two-dimensional process $(n,S(n))_{n\in\N}$ is a time-homogeneous Markov chain. 
   For every even integer $k\geq 0$, we denote  the law of $S$ started at time $k$  from $0$ by $\P_k$,
   that is
 \begin{equation}\label{E:P_k} \P_k \text{ is the conditional distribution of }(S(n))_{n\geq k} \text{ given  }S(k)=0.
 \end{equation}

 
We next introduce the sequence of factors
\begin{equation}\label{E:defa}
a_0\coloneqq 0 \text{ and }a_n\coloneqq \frac{\Gamma(n)}{\Gamma(n+2p-1)}, 
\qquad \text{for }n\geq 1,
\end{equation}
with the convention that $a_1=0$ for $p=0$, and  
recall from Stirling formula the useful estimate
\begin{equation}\label{E:estab}
a_n \sim n^{1-2p}  \qquad \text{ as }n\to \infty. 
\end{equation}
We then define the process
$$M_k(n)\coloneqq a_{k+n} S(n+k)
\qquad \text{for }n\geq 0,$$
and gather a few of its basic properties. We provide proofs for the reader's convenience, though closely related results have been already observed in the literature.

\begin{lemma} \label{L0}  For every even integer $k\geq 0$, we have:
\begin{enumerate}
\item[(i)]
The process $M_k=(M_k(n))_{n\geq 0}$  
is a $\P_k$-martingale.
\item[(ii)] Its $(n+1)$-th increment 
$$\Delta M_k(n+1)\coloneqq M_k(n+1)-M_k(n)$$ 
satisfies
$$\Delta M_k(n+1)= \frac{1-2p}{k+n+2p-1} M_k(n) \pm a_{k+n+1}.$$
\item[(iii)]
For every $q\geq 1$, one has 
$$\E(|M_k(n)|^{2q}) \leq c_q (k+n)^{(3-4p)q}.$$

\end{enumerate}
\end{lemma}
\begin{proof}Indeed, for $k=0$, the martingale property of $M_0$ was observed in \cite{Heyde}, see also \cite{Bercu,ColGavSch1}, and the extension for arbitrary $k\geq 2$ follows immediately. Next, since $S$ is a nearest neighbor process on $\Z$, we have
$$\Delta M_k(n+1)= a_{k+n+1}(S(n+k) \pm 1) - a_{k+n} S(n+k),$$
and observing the identity 
\begin{equation}\label{E:eqrefer}a_{n+k+1}-a_{n+k}= a_{n+k} \frac{1-2p}{n+k+2p-1},
\end{equation}
this yields (ii). 

We finally stress that, since plainly $|M_k(n)| \leq a_{n+k}(n+k)$ and $|1-2p| \leq 1$, one has 
\begin{equation} \label{E:elemb}
\left |\frac{1-2p}{k+n+2p-1} M_k(n)\right|\ \leq a_{k+n+1}.
\end{equation}
We then see from (ii)  that 
$|\Delta M_k(n+1)|\leq 2 a_{k+n+1}$.
The quadratic variation of $M_k$ is hence bounded from above  by $[M_k,M_k](n) \leq 4 A_{n+k}$, with
\begin{equation}\label{E:defA}
A_{n}\coloneqq \sum_{j=1}^n a^2_{j},
\end{equation}
Note from \eqref{E:estab} that
\begin{equation}\label{E:estab'}
A_{n}\sim \frac{n^{3-4p}}{3-4p}  \qquad \text{ as }n\to \infty;
\end{equation}
the last claim now follows from the Burkholder-Davis-Gundy inequality.
\end{proof}

 The  idea at the basis of this work is borrowed from \cite{ColGavSch2}. We shall embed the $\P_k$-martingales  $M_k$ associated to the elephant random walk in Lemma \ref{L0} into a linear Brownian motion $B=(B(t))_{t\geq 0}$ started from $B(0)=0$. That is, for each even number $k$,  we shall construct an increasing sequence $(T_{k,n})_{n\geq 0}$ of stopping times for $B$ such that  
 \begin{equation}\label{E:embedding}
 \text{the processes }
 (M_k(n))_{n\geq 0} \text{ and } (B(T_{k,n}))_{n\geq 0} \text{ have the same law.}
 \end{equation} 
 More specifically, it is implicit  that the process $M_k$ above is considered under the conditional probability measure $\P_k$.

Recall from Lemma \ref{L0}(ii) that $M_k$ is a binary splitting martingale, in the sense that given its past up to time $n$, its next increment $\Delta M_k(n+1)$ can only take two values. There is a basic embedding of any binary splitting martingale in a Brownian path that we now describe.
We define
inductively the increasing sequence of stopping times $(T_{k,n})_{n\geq 0}$ by $T_{k,0}\coloneqq 0$, and for $n\geq 0$, $T_{k,n+1}$ by 
\begin{equation} \label{E:defT}
 \inf\left\{t> T_{k,n}: B(t)-B(T_{k,n})= \frac{1-2p}{k+n+2p-1} B(T_{k,n}) \pm a_{k+n+1}\right\},
\end{equation}
where we agree that $T_{0,1}=0$ for $p=1/2$ and $k=0$. Then \eqref{E:embedding}
holds, and without loss of generality, we shall henceforth assume that the elephant random walk and the Brownian motion have been constructed in such a way that  the identity \eqref{E:embedding}
actually holds a.s. and not merely in distribution. 

The embedding yields a simple connection between the zeros of $S$  and the passage times  of $B$ at $0$ that lies at the heart of this work.
In this direction, it is convenient to introduce the notation
$$\mathcal{Z}_k\coloneqq \{n\geq 0: S(k+n)=0\}.$$

 \begin{lemma} \label{L5pre}  For every $k\geq 0$ ($k\geq 2$ for $p=0$), the inclusions
$$\{T_{k,n}: n\in \mathcal{Z}_k\} \subset \{t\geq 0: B(t)=0\} \subset \bigcup_{n\in \mathcal{Z}_k} [T_{k,n}, T_{k,n+1})$$
hold $\P_k$-a.s.
 \end{lemma}

\begin{proof}    Plainly, there are  the identities
 $$\mathcal{Z}_k= \{n\geq 0: M_k(n)=0\} = \{n\geq 0: B(T_{k,n})=0\}$$
 (recall that we impose $k\geq 2$ when $p=0$). 
We next observe that the embedding can be rephrased as follows. 
For every $n\geq 0$, one has $M_k(n)=B(T_{k,n})\in a_{k+n} \Z$, say $B(T_{k,n})=a_{k+n}z$ for some $z\in \Z$. 
We then consider the open interval
$$
I_{k,n+1}\coloneqq ((z-1)a_{k+n+1}, (z+1)a_{k+n+1}).
$$
Note from \eqref{E:eqrefer} that $a_{n+k}z\in I_{k,n+1}$.
During the time interval $[T_{k,n}, T_{k,n+1})$, the Brownian trajectory is confined in $I_{k,n+1}$   and hits one of the two extremities of that interval  at time $T_{k,n+1}$. 
The probability of an exit via the upper endpoint  is
$$\frac{1}{2} + \frac{2p-1}{2}\cdot \frac{z}{k+n},$$
and we recover of course the transition of the elephant random walk.

We stress that the origin lies in $I_{k,n+1}$ if and only if $B(T_{k,n})=0$, and then $I_{k,n+1}=(-a_{k+n+1},a_{k+n+1})$. 
Note also that $0$ is a boundary point of $I_{k,n+1}$ if and only if $|z|=1$, that is, if and only if $S(k+n)=\pm 1$. 
 This yields the inclusions in the statement. 
 \end{proof} 

\section{Scaling limit for the counting process of zeros}
The purpose of this section is to establish a precise version of \eqref{E:asympZ}, namely Theorem \ref{TZ} below. 
For the sake of simplicity, we will assume throughout this section that $p>0$\footnote{The case $p=0$ would request only a few very minor and purely technical modifications of the argument related to  the fact that then $S(2)=0=a_1$.}. We work here  with the unconditioned elephant (i.e. started at time $k=0$), and for the sake of simplicity, we henceforth omit the subscript $k$ in previous notation. In particular we will write $T_n=T_{0,n}$, $\mathcal{Z}= \mathcal{Z}_0$, $\P=\P_0$,  etc. 
Recall also from \eqref{E:defZ}  that 
$$Z(n)=\mathrm{Card}(\{1\leq j \leq n: j\in \mathcal{Z}\})$$
 counts the number of returns to the origin made by the elephant before time $n$.

We first need to introduce some notation.
We write $L=(L(t))_{t\geq 0}$ for the process of the local time  at level $0$ of the Brownian motion $B$, with the usual normalization so that $|B|-L$ is again a Brownian motion.
Recall  that $L$ has continuous non-decreasing sample paths such that the support of the Stieltjes measure $\dd L(t)$ coincides with the zero set of $B$, and that
$$\E(L(t))=\sqrt{2t/\pi} \qquad \text{for all }t\geq 0.$$
Consider  the process of Stieltjes integrals
$$F(t)\coloneqq \int_0^t f(s) \dd L(s), \qquad t\geq 0,$$
for any measurable function $f: \R_+\to \R_+$,
and observe from the Fubini-Tonelli theorem  that
$\E(F(t))<\infty$ if and only if $\int_0^t f(s) s^{-1/2} \dd s <\infty$. The latter holds in particular for any $t\geq 0$ when 
$f(s)=s^{\gamma}$ is a power function with exponent $\gamma>-1/2$, and \textit{a fortiori} for $\gamma =(2p-1)/(3-4p)\geq -1/3$.

We then define the process 
\begin{equation} \label{E:defLambda}
{H}(t) \coloneqq \frac{1}{\sqrt{3-4p}}  \int_0^{t^{3-4p}}  s^{(2p-1)/(3-4p)} \dd L(s), \qquad t\geq 0,
\end{equation}
and can now state our first main result.

\begin{theorem} \label{TZ} There is the weak convergence in the sense of finite dimensional distributions
$$\lim_{n\to \infty} \left(\frac{1}{\sqrt n} Z(nt)\right)_{t\geq 0} \egaldistr  \left({H}(t)\right)_{t\geq 0}.$$
\end{theorem}

\begin{remark}\label{Rema}  An easy computation based on the Fubini theorem shows that $\E({H}(t))=\sqrt{(6-8p)t/\pi}$, which is a decreasing function of the memory parameter $p$. This observation modulates a comment made in the introduction. Even though the memory parameter $p$ does not affect the growth exponent of the number of zeros of an elephant random walk, it does impact the scaling limit. On average, the latter decays as $p$ increases,  as it should expected from the dynamics. 
\end{remark}

The rest of this section is devoted to the proof of Theorem \ref{TZ}, which relies crucially on Brownian excursion theory. Our approach is related to David Williams' slick 
argument for establishing L\'evy's downcrossing Theorem, see \cite{Williams} and \cite[Exercise 1.19 on page 233]{RY}.

Recall that a compact time interval $[\ell,r]$ with $0< \ell < r$ is said to be an excursion interval (of $B$ aways from $0$) if and only $B(\ell)=B(r)=0$ and $B(t)\neq 0$ for all $t\in(\ell,r)$. 
We deduce from Lemma \ref{L5pre} that for any excursion interval $[\ell,r]$, either its right-extremity is given by $r=T_j$ for some integer  $j\geq 1$ in $\mathcal{ Z}$ (i.e. $j$ is a zero of $S$), in which case we say that $[\ell,r]$ \textit{counts},  or the whole excursion interval $[\ell,r]$ is contained in an open interval $(T_j, T_{j+1})$ for some $j\in \mathcal{ Z}$, in which case we say that $[\ell,r]$ \textit{does not count} (observe that, since each $T_j$ is a stopping time, the strong Markov property entails that if $B(T_j)=0$, then $B$ returns to $0$ immediately after time $T_j$ and hence $T_j$ is never the left-extremity of an excursion interval). This enables us to identify $Z(n)$ as the number of excursion intervals that count and are included in $[0,T_n]$.

In order to resolve the alternative of whether an excursion interval counts or not, we consider the partition of $\R_+$ induced by the sequence of stopping times $(T_n)_{n\geq 0}$
and define a right-continuous step process $(\alpha(t))_{t\geq 0}$  such that
$$\alpha(t)=a_{n+1} \qquad \text{for all } t\in[T_n, T_{n+1}).$$
Plainly, $\alpha$ is adapted to the Brownian filtration. 
Observe from the very construction of the embedding \eqref{E:defT}  that for every $j\in \mathcal{ Z}$ and every excursion interval $[\ell,r]\subset (T_j, T_{j+1})$, 
we have $\alpha(\ell)=a_{j+1}$ and  $|B(t)|<\alpha(\ell)$ for all $t\in [\ell,r]$.
Conversely, if $[\ell,r]$ is an excursion interval with $\max_{t\in[\ell,r]}|B(t)| \geq  \alpha(\ell)$, then there exists no $j\in \mathcal{ Z}$
such that $[\ell,r]\subset (T_j, T_{j+1})$, and thus $[\ell,r]$ counts.
We can now summarize this discussion in the following statement.

\begin{lemma} \label{Lcount} For every $n\geq 0$, $Z(n)$ coincides with the number of excursion intervals $[\ell,r]$ in $[0,T_n]$ such that 
$\max_{t\in[\ell,r]}|B(t)| \geq  \alpha(\ell)$. 
\end{lemma}

We shall also need  the following estimates for the asymptotic behavior of $T_n$ and $\alpha(t)$ which we essentially  lift from \cite{ColGavSch2}.

\begin{lemma} \label{Lassalpha} The following asymptotic  equivalences hold $\P$-almost-surely:
\begin{enumerate}
\item[(i)] $T_n \sim n^{3-4p}/(3-4p)$, as ${n\to \infty}$,
\item[(ii)] $\alpha(t)  \sim  ((3-4p)t)^{(1-2p)/(3-4p)}$ as ${t\to \infty}$.
\end{enumerate}
\end{lemma}
\begin{proof} The first limit is a translation of \cite[Equation (12)]{ColGavSch2}. Then define the inverse map $T^{-1}: \R_+\to \N$ by   $T^{-1}(t)=n+1$ for $T_n\leq t < T_{n+1}$. We deduce from (i) that 
$$ T^{-1}(t) \sim ((3-4p)t)^{1/(3-4p)} \qquad \text{a.s. as }{t\to \infty}$$
The second limit now follows from the identity $\alpha(t)=a_{T^{-1}(t)}$ and \eqref{E:estab}. 
\end{proof}

We continue by recalling some elements of It\^{o}'s excursion theory which will be useful to the proof of Theorem \ref{TZ}, referring to Chapter XII in \cite{RY} for background. 
The (right-continuous) inverse local time process
$$\lambda(s)\coloneqq \inf\{t\geq 0: L(t)>s\}, \qquad s\geq 0,$$
is a stable subordinator with index $1/2$; it enables us to identify the family of the excursion intervals of $B$ as
$$\left\{\left[ \lambda(s-), \lambda(s)\right]: s>0 \text{ with } \Delta \lambda(s)\coloneqq \lambda(s)- \lambda(s-)>0\right\}.$$
We write 
$$h_s\coloneqq \max_{t\in [\lambda(s-), \lambda(s)] } |B(t)|$$
for the (absolute) height of such excursions. It is then a well-known fact from It\^{o}'s excursion theory that
 $\{(s, h_s):  \Delta \lambda(s)>0\}$ is the family of the atoms of a Poisson point process on $\R_+\times \R_+$ with intensity $h^{-2} \dd s \dd h$. 
 
Excursion theory incites us to introduce the counting process
$$ \nu(t)\coloneqq  \mathrm{Card}(\{0<s\leq t: h_s\geq \alpha(\lambda(s-)\}), \qquad t\geq 0.$$
 
 \begin{lemma}\label{Lasszeta} We have with probability one that
 $$ \nu(t) \sim \int_0^t \frac{ \dd s}{\alpha(\lambda(s))} \qquad \text{as }t\to \infty.$$
 \end{lemma}
\begin{proof}
The elements of excursion theory which we recalled above entail that
 $$ \nu^c(t) \coloneqq \int_0^t \frac{ \dd s}{\alpha(\lambda(s))}, \qquad t\geq 0,$$
 is the so-called compensator of the counting process $\nu$. It is readily seen from Lemma \ref{Lassalpha}(ii)  and the fact that the inverse local time $\lambda$ is a stable subordinator
 with index $1/2$ that $\lim_{t\to \infty} \nu^c(t)=\infty$ a.s. 
 By a classical result due to Meyer \cite{Meyer} and Papangelou \cite{Papangelou} (see also Corollary 25.26 in \cite{Kal}), we then know that the counting process $\nu$ can be seen as some standard Poisson process time-changed by $\nu^c$. Our claim then follows from the law of large numbers for the Poisson process.
\end{proof}

\begin{proof}[Proof of Theorem \ref{TZ}]
In the previous notation, Lemma \ref{Lcount} translates into the identity
$Z(nt) = \nu(L(T_{\lfloor nt\rfloor}))$, and then Lemma \ref{Lasszeta} yields
$$Z(nt)\sim \int_0^{L(T_{\lfloor nt\rfloor})} \frac{ \dd s}{\alpha(\lambda(s))} =  \int_0^{T_{\lfloor nt\rfloor}} \frac{ \dd L(s)}{\alpha(s)}.$$
More precisely, the second equality follows from the change of variable formula 
\begin{equation}\label{E:change}
\int_0^t f(s) \dd L(s)=\int_0^{L(t)} f(\lambda(u)) \dd u \qquad \text{a.s.}
\end{equation}
where $f:\R_+\to \R_+$ stands  for a generic measurable function. Indeed, $L$ being continuous, one has $L\circ \lambda = \mathrm{Id}$,
and one deduces by a monotone class argument that the Lebesgue measure $\dd u$  on $\R_+$ is the pushforward measure of the Stieltjes measure $\dd L(s)$
by the map $\lambda$. (Beware that the role of $L$ and $\lambda$ cannot be interchanged, the Lebesgue measure $\dd u$ is not the pushforward measure of  $\dd \lambda(s)$ by $L$.)

Next we readily deduce from Lemma \ref{Lassalpha}(i) and an argument of monotonicity that as $n\to \infty$, 
$$ \int_0^{T_{\lfloor nt\rfloor}} \frac{ \dd L(s)}{\alpha(s)} \sim  \int_0^{(tn)^{3-4p}/(3-4p)} \frac{ \dd L(s)}{\alpha(s)} 
= \int_0^{t^{3-4p}} \frac{ \dd L(sn^{3-4p}/(3-4p))}{\alpha(sn^{3-4p}/(3-4p))}.$$
On the one hand, we also know from Lemma \ref{Lassalpha}(ii) that almost surely
$$\alpha(sn^{3-4p}/(3-4p)) \sim n^{1-2p} s ^{(1-2p)/(3-4p)} \qquad \text{as }n\to \infty.$$
On the other hand,  the Brownian scaling property entails that there is  the identity in distribution 
$$\left( L(sn^{3-4p}/(3-4p))\right)_{s\geq 0} \egaldistr \sqrt{\frac{n^{3-4p}}{3-4p}} (L(s))_{s \geq 0}.$$

We stress that $(1-2p)/(3-4p)\leq 1/3<1/2$. It follows easily from dominated convergence that as $n\to \infty$, there is the convergence in the sense of finite dimensional  distributions (in the parameter $t$)
$$ n^{-1/2}  \int_0^{t^{3-4p}} \frac{ \dd L(sn^{3-4p}/(3-4p))}{\alpha(sn^{3-4p}/(3-4p))} \Longrightarrow   \frac{1}{\sqrt{3-4p}}  \int_0^{t^{3-4p}}  s^{(2p-1)/(3-4p)} \dd L(s).$$
Putting the pieces together, this completes the proof of our claim. 
\end{proof}

We now conclude this section with some comments about 
a related but easier result, namely the  convergence in distribution  on the space of closed subsets of $\R_+$ endowed with the so-called Fell-Matheron topology:
\begin{equation}\label{E:Fell}
\lim_{n\to \infty} \{n^{-1}j^{3-4p}: j\in \mathcal Z\} \egaldistr \{t\geq 0: B(t)=0\}.
\end{equation}
This claim can  easily be deduced  from the  estimate for $T_n$ of Lemma \ref{Lassalpha}(i), the inclusions of Lemma \ref{L5pre}
and the Brownian scaling property. It can also be deduced from the scaling limit theorem for the elephant itself
$$\lim_{n\to \infty} \frac{1}{\sqrt n} \left( S(\lfloor nt \rfloor)\right)_{t\geq 0} \egaldistr  \frac{1}{\sqrt{3-4p}} \left( t^{2p-1} B(t^{3-4p})\right)_{t\geq 0},$$
see \cite[Theorem 1]{BaurBer}.
 We stress that although Theorem \ref{TZ} and \eqref{E:Fell} are obviously related, one cannot deduce the former from the latter as the functional that counts elements in closed subsets is not continuous for the Fell-Matheron topology. Actually \eqref{E:Fell} could even be misleading as it might suggests that the weak limit in Theorem \ref{TZ} should be the process
$\left(L(t^{1/(3-4p)})\right)_{t\geq 0}$   (because, roughly speaking, the Brownian local time process $L$ is the natural measure on the zero set of $B$), which is false, except of course for $p=1/2$ when the elephant random walk coincides with the usual simple symmetric random walk.
We refer to \cite[Appendix C]{Molchanov} for background on the Fell-Matheron topology and leave details of the proof of \eqref{E:Fell} to the interested reader; see also the forthcoming Remark \ref{R:Fell2}.

\section{Uniform tail estimates for return times}

Recall from \eqref{E:P_k}  that  $\P_k$ denotes the law of an elephant random walk started at time $k$ from $0$; let also 
\begin{equation} \label{E:defRk}
{R}\coloneqq \inf\{n\geq 1: S(n+k)=0\}
\end{equation} stand for the amount of  time the latter needs to first return to the origin. 
The purpose of this section is to establish uniform asymptotic estimates for the tail distributions of $R$ under $\P_k$.
Recall the definitions   \eqref{E:defa} and \eqref{E:defA},  and the estimates \eqref{E:estab}  and \eqref{E:estab'}.

\begin{theorem}\label{T1} For every fixed $b>0$, the convergence 
$$\lim_{n\to \infty}  \frac{\sqrt{A_{n+k}-A_{k}}}{a_{k+1}} \P_k({R}>n) =\sqrt{\frac{2}{\pi}}$$
holds uniformly in $k\leq bn$ (and $k\geq 2$ when $p=0$). 
\end{theorem}
Before tackling the proof of Theorem \ref{T1}, let us point out that, thanks to \eqref{E:estab'}, our earlier claim \eqref{E:T1} is merely a special case of  the latter for $k=0$. 
The proof of Theorem \ref{T1} relies crucially on uniform concentration estimates for  the distribution of the stopping times $T_{k,n}$ which
enable to embed the elephant random walk into a Brownian path in Section 2. These  are developed in the next section.

\subsection{Concentration estimates for the embedding}

Recall that the sequence $(A_n)$ has been defined in \eqref{E:defA}, and from Lemma \ref{Lassalpha}(i) that for $k=0$,  one has $T_{0,n}=T_n\sim A_n$. The purpose of this section is to establish the following uniform bound for the deviations. 
 We use here the notation $\P$ for the law of the Brownian motion $B$, which obviously does not depend on $k$ (but of course the embedding $(T_{k,n})_{n\geq 0}$ does). 

\begin{proposition} \label{P1}  For every $\varepsilon >0$, $r\geq 1$, and even number $k\geq 0$, one has:
$$ \P\left( \sup_{ \ell \leq n}\left | T_{k,\ell} -(A_{k+\ell}-A_k) \right| \geq  \varepsilon (k+n)^{3-4p} \right) \leq c_{\varepsilon,r} (k+n)^{-r}.$$
\end{proposition}
The rest of this section is devoted to the proof of Proposition \ref{P1}. 
We start with a couple of elementary observations. 
\begin{lemma} \label{L:estab} The following assertions hold:
\begin{itemize}
\item[(i)]  There is the convergence
$$\lim_{n\to \infty} \frac{A_{k+n}-A_k }{a_{k+1}^{2}}=\infty \qquad\text{uniformly in }k. $$ 
\item[(ii)] For every $b<\infty$, there exists some $c>0$ such that for all integers $k\leq bn$ ($k\geq 2$ if $p=0)$ and $n\geq 1$:
$$ A_{k+n}-A_k\geq c (n+k)^{3-4p}.$$
\end{itemize} 
 \end{lemma}
\begin{proof}  (i) We start writing
$$ \frac{A_{k+n}-A_k }{a_{k+1}^{2}}= \sum_{j=1}^n \left( \frac{a_{k+j}}{a_{k+1}}\right)^2,$$
and observe that for $j\geq 2$, 
$$\frac{a_{k+j}}{a_{k+1}} = \frac{ (k+1)}{(k+2p)}  \cdots \frac{(k+j-1)}{(k+2p+j-2)}.$$
In the case $p\leq 1/2$, each quotient in the product in the right-hand side is at least $1$, so we have $a_{k+j}/a_{k+1}\geq 1$, which immediately gives the claim. 
In the case $\frac{1}{2}< p < \frac{3}{4}$, the ratio  $(\ell+1)/(\ell+2p)= 1-(2p-1)/(\ell+2p)$ increases with $\ell$, thus $a_{k+j}/a_{k+1}$ increases with $k$, and this
yields the bound $a_{k+j}/a_{k+1} \geq a_j/a_1$. 
As a consequence, we have then 
$$ \frac{A_{k+n}-A_k }{a_{k+1}^{2}}\geq a_1^{-2} A_n,$$
which also yields our claim thanks to \eqref{E:estab'}.

(ii)  We argue by contradiction. Suppose that the claim fails. There would exist some integer sequence $(k(n))_{n\geq 1}$ with $\sup_{n\geq 1} k(n)/n < \infty$
and 
$$\liminf_{n\to \infty}\frac{A_{k(n)+n}-A_{k(n)}}{(k(n)+n)^{3-4p}}=0.$$
We could then excerpt some subsequence along which $k(n)/n$ converges to say $\beta$
and further $(A_{k(n)+n}-A_{k(n)})n^{4p-3}$ tends to $0$. But we know from \eqref{E:estab} and \eqref{E:estab'} that the latter converges then to 
$(1+\beta)^{3-4p}-\beta^{3-4p}>0$, 
in contradiction with the preceding. 
\end{proof}

The iterative definition \eqref{E:defT} of the stopping times $T_{k,n}$ incites us to introduce the notation 
$$\tau(x,y)\coloneqq \inf\{t\geq 0: |B(t)+y|=x\}$$
for the first exit time from the interval $(-x,x)$ by the Brownian motion $y+B$ started from $y$,
where $x>0$ and $y\in(-x,x)$. We first point at the following basic facts.

\begin{lemma} \label{L1prep} We have:
\begin{enumerate}
\item [(i)] 
For every $x>0$, 
there is the identity in distribution
$$\tau(x,0) \egaldistr x^2 \tau(1,0).$$
Furthermore, we have $\E(\tau(1,0)^q) < \infty$ for all $q>0$.

\item [(ii)] For every  $y\in(-x,x)$, the variable $\tau(x,y)$ is dominated stochastically by $\tau(x,0)$, and we also have
$$\E(\tau(x,y))= x^2-y^2.$$
\end{enumerate}
\end{lemma}
\begin{proof} (i) The first assertion is plain  from
the scaling property, and the second from the well-known fact that  $\E(\e^{r\tau(1,0)})<\infty$ for any $r< \pi^2/8$.

(ii) Observe that $\tau(x,y)$ has the distribution of the first hitting time of $x$
by a reflected Brownian motion on $\R_+$ started from $|y|$. By the strong Markov and the symmetry properties of Brownian motion, this yields the identity
$$\tau(x,0)\egaldistr \tau(|y|,0) + \tau(x,y)$$
where  the two variables in the sum in the right-hand side are independent. The stochastic domination is now clear, and so is the final assertion,
since $\E(\tau(x,0))= x^{2}$.
\end{proof}

 We write $(\mathcal F_t)_{t\geq 0}$ for the natural filtration generated by the Brownian motion and consider the increments 
 $$\Delta T_{k,n+1}\coloneqq T_{k,n+1}-T_{k,n}, \qquad n\geq 0.$$ 
 The strong Markov property and the definition \eqref{E:defT} entail that on the event  $\{B(T_{k,n})=b\}$  for some $b\in a_n\Z$, 
  the conditional distribution of $\Delta T_{k,n+1}$ given $\mathcal F_{T_{k,n}}$  is that of 
 $$\tau\left( a_{k+n+1},-\frac{2p-1}{k+n+2p-1}b  \right) .$$ 
This incites us to introduce  
 $$V_k(n)\coloneqq   \sum_{j=1}^{n-1} \left(\frac{2p-1}{k+j+2p-1}\right)^2 B^2(T_{k,j}).$$
  We shall  need the following asymptotic bound for the moments of $V_k(n)$.
\begin{lemma} \label{L3}
For every integer ${m}\geq 1$, there is the inequality
$$\E\left( V_k(n)^{m} \right) \leq \left\{ 
\begin{matrix} c_m(n+k)^{(2-4p){m}} & \text{ if } p<1/2,\\
c_m & \text{ if } p\geq 1/2.
\end{matrix} 
\right.
$$
\end{lemma}
\begin{proof} Since $V_k(n)\equiv 0$ for $p=1/2$, we focus on the case $p\neq 1/2$. 
Using first H\"older inequality and then the $q=m$ case of Lemma \ref{L0}(iii), we get
\begin{align*}
\E\left( V_k(n)^{m} \right) &= (2p-1)^{2{m}}\sum_{j_1, \cdots, j_{{m}}=1}^{n-1} \frac{\E( B^2(T_{k,j_1})\cdots B^2(T_{k,j_{{m}}}))}{ (k+j_1+2p-1)^2\cdots (k+j_{{m}}+2p-1)^2} \\
&\leq  \sum_{j_1, \cdots, j_{{m}}=1}^{n-1} \frac{\left(\E( B^{2{m}}(T_{k,j_1}))\cdots \E(B^{2{m}}(T_{k,j_{{m}}}))\right)^{1/{m}}}{ (k+j_1+2p-1)^2\cdots (k+j_{{m}}+2p-1)^2} \\
&\leq c_m \sum_{j_1, \cdots, j_{{m}}=1}^{n-1} \frac{ (k+j_1)^{3-4p}\cdots (k+j_{{m}})^{3-4p}}{ (k+j_1+2p-1)^2\cdots (k+j_{{m}}+2p-1)^2}  \\
&\leq c_m \sum_{j_1, \cdots, j_{{m}}=1}^{n-1}  (k+j_1)^{1-4p}\cdots (k+j_{{m}})^{1-4p} \\
&\leq c_m \left( \sum_{j=1}^{n-1} (k+j)^{1-4p}\right)^{{m}}.
\end{align*}
This proves our claim. 
\end{proof}
In the notation introduced right before Lemma \ref{L3}, we have on the event  $\{B(T_{k,n})=b\}$   for some $b\in a_n\Z$,   that
\begin{align*}\E( \Delta T_{k,n+1}\mid \mathcal F_{T_{k,n}}) &= \E\left( \tau\left( a_{k+n+1},-\frac{2p-1}{k+n+2p-1}b  \right)\right)\\
&= a_{k+n+1}^2-  \left(\frac{2p-1}{k+n+2p-1}\right)^2 B^2(T_{k,n}),
\end{align*}
where we used Lemma \ref{L1prep}(ii) for the second equality. Hence  the compensated sum 
 $$N_k(n)\coloneqq T_{k,n}-  (A_{k+n}-A_k) + V_k(n), \qquad n\geq 0,$$
 is a martingale.  We  point at the following upperbound, which is the cornerstone of our analysis.

 \begin{lemma} \label{L2} For every $q\geq 1$, we have
$$\E\left(\sup_{1\leq \ell \leq n}N_k(\ell)^{2q}\right) \leq  c_q\left(\sum_{j=1}^n a_{k+j}^4\right)^q.$$
\end{lemma} 

 \begin{proof} We write 
 $$[N_k,N_k](n)\coloneqq \sum_{j=1}^n (N_k(j)-N_k(j-1))^2$$
 for the quadratic variation of the martingale $N_k$ and
   first claim that for each fixed even number $k\geq 0$, there exists a sequence  $(\eta_n)_{n\geq 1}$ of i.i.d. variables on some enlarged probability space, where each $\eta_n$ has the same law as $1+\tau^2(1,0)$, and  such that for all $n\geq 1$
 \begin{equation} \label{E:domsto}
 [N_k,N_k](n) \leq \sum_{j=1}^n a_{k+j}^4 \eta_j . 
 \end{equation}
Indeed, recall from \eqref{E:elemb} that 
$$\left ( \frac{2p-1}{k+n+2p-1}\right )^2 B^2(T_{k,n}) \leq  a^2_{k+n+1};$$
this yields   the simple bound
 $$(N_k(n+1)-N_k(n))^2 \leq  a_{k+n+1}^4 + |\Delta T_{k,n+1}|^2.$$
 Then recall also that the conditional law of $\Delta T_{k,n+1}$ given $\mathcal F_{T_{k,n}}$
is that of $\tau\left( a_{k+n+1},y\right)$ with $y=-\frac{2p-1}{k+n+2p-1}B(T_{k,n})$.   An application of Lemma \ref{L1prep}  shows that there is a variable, say $\xi_{n+1}$, which dominates $\Delta T_{k,n+1}$, and which conditionally given $\mathcal F_{T_{k,n}}$, is
distributed as $a_{k+n+1}^2\tau(1,0)$. We stress that the conditional distribution of $\xi_{n+1}$ does not depend on $B(T_{k,n})$, and hence $\xi_{n+1}$ is independent of $\mathcal F_{T_{k,n}}$ (even though, clearly $\xi_{n+1}$ depends on $\Delta T_{k,n+1}$). Setting $\eta_{n+1}=1+\xi^2_{n+1}$ yields \eqref{E:domsto} by induction. 

 Combining \eqref{E:domsto} and the Burkholder-Davis-Gundy inequality, we get that
 $$\E\left(\sup_{1\leq \ell \leq n}|N_k(\ell)|^{2q}\right) \leq  c_q \E\left( \left( \sum_{j=1}^n a_{k+j}^4 \eta_j \right)^q\right).$$ 
 Since the variables $a_{k+j}^4 \eta_j$ are nonnegative and independent, we know from Corollary 3 of Latała \cite{Latala} that
  \begin{align*}\E\left( \left( \sum_{j=1}^n a_{k+j}^4 \eta_j \right)^q\right) &\leq c_q \left( \left( \sum_{j=1}^n a_{k+j}^4 \E(\eta_j) \right)^q  + \sum_{j=1}^n a_{k+j}^{4q} \E(\eta^q_j) \right) \\
  &\leq c_q \left( \sum_{j=1}^n a_{k+j}^4  \right)^q ,
    \end{align*}
  where for the last line, we used the fact that $\E(\eta^q_j)=\E((1+\tau(1,0))^q)<\infty$, see Lemma \ref{L1prep}(ii).
 \end{proof}
 We now have all the ingredients needed for the proof of Proposition \ref{P1}.
 \begin{proof}[Proof of Proposition \ref{P1}] 
 Recall  \eqref{E:estab} and note that 
 $$ \sum_{j=1}^n a_{k+j}^4 \leq  \left\{ 
\begin{matrix}c (n+k)^{5-8p} & \text{ if } p<5/8,\\
c \log (n+k) & \text{ if } p=5/8,\\
c  & \text{ if } p>5/8.
\end{matrix} \right.$$
Then,  writing
$$\sup_{ \ell \leq n}\left | T_{k,\ell}-(A_{k+\ell}-A_k)  \right| \leq  \sup_{ \ell \leq n}|N_k(\ell)| + V_k(n),$$
and appealing to   Lemmas  \ref{L3} and \ref{L2}, we get
that for every   $m \geq 1$:
 $$ \E\left(\sup_{\ell \leq n} \left | T_{k,\ell}- (A_{k+\ell}-A_k) \right |^{2m} \right) \leq \left\{ 
\begin{matrix} c_m(k+n)^{(5-8p)m} & \text{ if } p<5/8,\\
c_m(\log (k+n))^{m} & \text{ if } p=5/8,\\
c_m& \text{ if } p> 5/8.
\end{matrix} 
\right. 
$$
It is then straightforward to complete the proof by the Markov inequality. 
\end{proof}

\subsection{Proof of Theorem \ref{T1}} 
Recall that we use $\P$ to denote the law of the Brownian motion $B$, and that for each even number $k$, the sequence of stopping times $(T_{k,n})_{n\geq 0}$ has been defined in Section 2 such that  
$$B(T_{k,n})=M_k(n)= a_{n+k}S(n+k)\qquad \text{ a.s.,}$$
 where in the right-hand side $(S(n))_{n\geq 0}$ has the law $\P_k$. We shall sometime use both $\P$ and $\P_k$
in the same equation when certain quantities involved  are expressed in terms of the Brownian motion only and some other in terms of the elephant random walk only.
In this setting, we have
$$T_{k,{R}} = \inf\{T_{k,n}: n\geq 1 \ \& \ S(n+k)=0\},$$
and the final ingredient  that we need for the proof of Theorem \ref{T1} is the following identity.
  
 \begin{lemma} \label{L5} The law of $a_{k+1}^{-2} T_{k,{R}}$ does not depend on $k$
 and one has
 $$\lim_{t\to \infty} \sqrt  t \P(T_{k,{R}} \geq  a_{k+1}^{2}  t) =  \sqrt {\frac{2}{\pi}}.$$
\end{lemma}

\begin{proof}  We have $R=\inf\{n\in \mathcal{Z}_k: n\geq 1\}$, and we infer from Lemma \ref{L5pre} that there is the identity
 $$T_{k,{R}}=\inf\{t>T_{k,1}: B_t=0\}.$$
Using the strong Markov property of Brownian motion, the fact that  $|B(T_{k,1})|= a_{k+1}$,  and Lemma \ref{L1prep}, we get the identity in distribution
$$T_{k,{R}} \egaldistr  a^2_{k+1} \left(\tau(1,0)+  \sigma\right),$$
where in the right-hand side, the two variables are independent and $\sigma\coloneqq \inf\{t\geq 0: B_t=1\}$ denotes the first hitting time of $1$ by the Brownian motion. The proof can now be completed with an appeal to the well-known estimate 
 $$\P(\sigma> s) \sim  \sqrt {\frac{2}{\pi s}} \qquad \text{as }s \to \infty,$$ and Lemma \ref{L1prep}(i).
 \end{proof} 
 
 We can now establish Theorem \ref{T1}.  On the one hand, we have  for every $s>0$ and $n\geq 0$ that
$$\{T_{k,{R}}\geq s\} \subset \{{R}\geq n\}\cup\{T_{k,n}\geq s\},$$
which yields the lower bound 
$$\P_{k}({R}\geq n) \geq  \P(T_{k,{R}}\geq s ) - \P(T_{k,n}\geq s ) .$$
We choose
$$s=s(k,n)\coloneqq (1+\varepsilon) (A_{n+k}-A_{k})$$ 
for some arbitrarily small $\varepsilon>0$, 
and recall from Lemma \ref{L5} that the function
$$t \mapsto \left| \sqrt  t \P(T_{k,{R}} \geq  a_{k+1}^{2}  t) -  \sqrt {\frac{2}{\pi}}\right|, \qquad t\geq 0,$$
does not depend on $k$ and has limit $0$ as $t\to \infty$. 
Recall also from Lemma \ref{L:estab}(i)
that  as $n\to \infty$, $a_{k+1}^{-2}(A_{n+k}-A_{k})$ converges to $\infty$ uniformly in $k$.
We deduce that
$$\lim_{n\to \infty} \frac{\sqrt{A_{n+k}-A_{k}}}{a_{k+1}}\P(T_{k,{R}} \geq  s(k,n)) =   \sqrt{\frac{2}{\pi (1+\varepsilon)}} \quad \text{uniformly in }k. $$
Next, provided that $k\leq bn$, we have from Lemma
\ref{L:estab}(ii) that for some $c>0$,
$$A_{n+k}-A_{k} \geq c (n+k)^{3-4p}.$$
We deduce  from Proposition \ref{P1} that
$$\lim_{n\to \infty} \frac{\sqrt{A_{n+k}-A_{k}}}{a_{k+1}}\,\P(T_{k,{n}} \geq  s(k,n))  =0 \quad \text{uniformly in }k\leq bn,$$
and conclude that
$$\liminf_{n\to \infty} \frac{\sqrt{A_{n+k}-A_{k}}}{a_{k+1}}\P_{k}({R}\geq n) \geq   \sqrt{\frac{2}{\pi (1+\varepsilon)}}  \quad \text{uniformly in }k\leq bn. $$

On the other hand, we have similarly for every $s>0$ and $n\geq 0$ that
$$ \{{R}\geq n\}\cap\{T_{k,n}\geq s\}\subset \{T_{k,{R}}\geq s\},$$
which yields the upper bound 
\begin{equation} \label{E:domin}
 \P_k({R}\geq n) \leq  \P(T_{k,{R}}\geq s) + \P(T_{k,n}< s ).
 \end{equation} 
We choose
$$s=s(k,n)\coloneqq (1-\varepsilon) (A_{n+k}-A_{k})$$ 
for some arbitrarily small $\varepsilon>0$, 
and get similarly from Lemma \ref{L:estab}, Proposition \ref{P1} and Lemma \ref{L5} that 
$$\limsup_{n\to \infty} \frac{\sqrt{A_{n+k}-A_{k}}}{a_{k+1}}\P_{k}({R}\geq n) \leq   \sqrt{\frac{2}{\pi (1-\varepsilon)}}  \quad \text{uniformly in }k\leq bn. $$
This completes the proof.

 \section{Reconciliation of the two limit theorems}
The purpose of this section is to resolve the apparent disagreement between Theorems \ref{TZ} and \ref{T1}  which has been exposed in the Introduction. In short,
we shall first observe that Theorem \ref{TZ} can be rephrased a scaling limit for the Markov chain of return times. Next, we shall point out that Theorem \ref{T1} is
actually  the cornerstone for the application a general result on scaling limit for integer valued Markov chains, and specifying the latter in our framework enables us to recover Theorem \ref{TZ}.

Recall that the process ${H}=({H}(t))_{t\geq 0}$  in Theorem \ref{TZ}  has been defined in \eqref{E:defLambda}. Plainly $H$ has continuous non-decreasing paths and we introduce its (right-continuous) inverse $\eta$, 
$$\eta(t)\coloneqq \inf\{s>0: {H}(s)>t\},\qquad t\geq 0.$$
Note from the Brownian scaling property that for every $c>0$, the rescaled process
$\left(c^{-1/2} {H}(ct)\right)_{t\geq 0}$ has the same distribution as $ {H}$.
Since clearly ${H}(1)>0$ a.s., self-similarity entails $\lim_{t\to \infty} {H}(t) =\infty$ a.s. and therefore $ \eta(t)<\infty$ for all $t\geq 0$ a.s. 
Moreover  the scaling property  can be transferred from $H$ to  $\eta$ and there is the identity in distribution 
\begin{equation}\label{E:scaling}
\left(c^{-2} \eta(ct)\right)_{t\geq 0} \egaldistr \eta.
\end{equation}

 We next point at a useful representation of  $\eta$ in terms of the inverse local time process $\lambda$. 
We first observe that random function
$$ t\mapsto \frac{1}{\sqrt{3-4p}}  \int_0^{t}  \lambda(s)^{(2p-1)/(3-4p)} \dd s, \qquad t\geq 0,$$
is bijective on $\R_+$ a.s. (recall that $\lambda$ is a stable subordinator with index $1/2$ and therefore, roughly speaking, $\lambda(s)\approx s^2$ both as $s\to 0+$ and as $s\to \infty$) and we define  implicitly its inverse function $\rho$ by
$$\frac{1}{\sqrt{3-4p}}  \int_0^{\rho(t)}  \lambda(s)^{(2p-1)/(3-4p)} \dd s = t ,  \qquad t\geq 0.$$

\begin{proposition} \label{P2} With probability one, there is the identity
$$\eta(t)= \lambda(\rho(t))^{1/(3-4p)}\qquad \text{for all } t\geq 0.$$
As a consequence, $\eta$ is a time-homogeneous strong Markov process.
\end{proposition}
\begin{proof} We start by observing from the change of variables formula \eqref{E:change} that  \eqref{E:defLambda} can be rewritten as
$$
{H}(t)= \frac{1}{\sqrt{3-4p}}  \int_0^{L(t^{3-4p})}  \lambda(s)^{(2p-1)/(3-4p)} \dd s, \qquad t\geq 0,
$$
from which we infer the identity
$$L\left(\eta(t)^{3-4p}\right)=\rho(t).$$
This yields 
$$ \lambda(\rho(t)-)^{1/(3-4p)}\leq \eta(t) \leq  \lambda(\rho(t))^{1/(3-4p)}, \qquad \text{for all } t\geq 0,$$
and since $\eta$ is right-continuous, we obtain the formula of the statement.

The subordinator $\lambda$ is a Feller process on $\R_+$, and the increasing process $\rho=(\rho(t))_{t\geq 0}$ has been defined as the inverse of a perfect continuous homogeneous additive functional of the latter. 
It follows that the time-changed process $\lambda\circ \rho$ is strongly Markovian, see e.g. \cite[Section III.21]{WillRo}. The same holds for
$\eta=(\lambda\circ \rho)^{1/(3-4p)}$ since the map $x\mapsto x^{1/(3-4p)}$ is bijective on $\R_+$.
 \end{proof}
 
 By Proposition \ref{P2} and \eqref{E:scaling}, the process $\eta$ is an increasing self-similar Markov process, and hence also a Feller process; see \cite{Lamperti} and \cite[Chapter 13]{Kypri}. Its infinitesimal generator $G_{\eta}$  can be computed on $(0,\infty)$ using Volkonskii's formula \cite[III.21.4 on page 277-8]{WillRo}. Indeed, it is well-known \cite[Eq. 6.7]{Lamperti}
that the  infinitesimal generator $G_{\lambda}$ of $\lambda$ is given, say for a smooth bounded function $f:\R_+\to \R$, by
 $$G_{\lambda}f(x)=  \sqrt{\frac{1}{2\pi x}} \int_1^{\infty} (f(xu)-f(x)) (u-1)^{-3/2}\dd y, \qquad \text{for }x>0,$$
and we then easily get
\begin{align}\label{E:Geta}
G_{\eta}f(x) &=  \sqrt{\frac{3-4p}{2\pi x}} \int_1^{\infty} (f(xu^{1/(3-4p)})-f(x)) (u-1)^{-3/2}\dd u\nonumber \\
&=\sqrt{\frac{(3-4p)^3}{2\pi x}} \int_1^{\infty} (f(x v )-f(x)) (v^{3-4p}-1)^{-3/2} v^{2-4p} \dd v.
\end{align}

We stress that $\eta$ starts from the entrance boundary point $0$; we shall need to consider as well its version (in the Markovian sense) started from $x>0$,
which we denote by $\eta_x= (\eta_x(t))_{t\geq 0}$. By the scaling property, $\eta_x$ has the same law as $\left(x\eta_1(t/\sqrt x))\right)_{t\geq 0}$. Furthermore,
 $\eta_x$ converges in distribution towards $\eta$ as $x\to 0+$, see  \cite{BeCa}.

Last but not least, a fundamental result of Lamperti \cite{Lamperti} identifies the logarithm of any self-similar Markov process on $(0,\infty)$ as the time-change of some real-valued L\'evy process; see also \cite[Section 13.3]{Kypri}. In the present case, 
$$Y(t)\coloneqq \log \eta_1(\varsigma(t)), \quad\text{where}\quad \int_0^{\varsigma(t)} \frac{\dd s}{\sqrt{\eta_1(s)}}=t, \qquad t\geq 0,$$
is a subordinator. 
Its L\'evy measure $\Pi$ is obtained as  the image by the map $v\mapsto x=\log v$ of the measure 
$$\sqrt{\frac{(3-4p)^3}{2\pi }} (v^{3-4p}-1)^{-3/2} v^{2-4p} \dd v, \qquad v>1,$$
that appears in the infinitesimal generator of the self-similar Markov process $\eta$ in \eqref{E:Geta}, and we get 
\begin{equation} \label{E:Pi}
\Pi(\dd x) = \sqrt{\frac{(3-4p)^3}{2\pi }} \left( \e^{(3-4p)x}-1\right)^{-3/2} \e^{(3-4p)x} \dd x, \qquad x>0. 
\end{equation} 
We further infer from \eqref{E:Geta} that $Y$ has no drift, that is the L\'evy-Khintchin formula reads
$$\E\left(\exp(-q Y(t))\right) =\exp\left(-t\int_0^{\infty}(1-\e^{-qx})\Pi(\dd x)\right), \qquad q\geq 0.$$

Next, we turn our attention back to the elephant random walk and the set $\mathcal Z$ of its zeros,  working under $\P=\P_0$.
We  enumerate the elements of $\mathcal Z$ in the increasing order, namely
$\zeta(0)=0 < \zeta(1)=R<\zeta(2) < \ldots$, so that $\zeta(j)$ is the time of the $j$-th return of the elephant to the origin and 
$$Z(n)< j \ \Longleftrightarrow \zeta(j)>n.$$
In this framework, Theorem \ref{TZ} can be rephrased  as follows:
there is the weak convergence in the sense of finite dimensional distributions of the rescaled sequence
\begin{equation}\label{E:TZ}\lim_{n\to \infty} \left(n^{-1} \zeta(\lfloor n^2t\rfloor)\right)_{t\geq 0} \egaldistr  \left(\eta(t)\right)_{t\geq 0}.
\end{equation}

As it has already been discussed in the introduction, Theorem \ref{TZ} does not seem to fit with Theorem \ref{T1}.  We shall now establish the following consequence of Theorem \ref{T1}  which actually fully agrees with \eqref{E:TZ}, hence reconciling our two main results. 
Recall that for every even integer, $\P_k$ stands for the law of the elephant started at time $k$ from $0$.  Under $\P_k$, we  use  the notation $(\zeta_k(j))_{j\geq 0}$ for the sequence that 
enumerates the set of its zeros $\mathcal{Z}_k$ in the increasing order; in particular  $\zeta_k(0)=k$ and $\zeta_k(1)=R+k$, $\P_k$-a.s. 

 \begin{corollary} The law of 
 $\left(n^{-1} \zeta_n(\lfloor n^2t\rfloor)\right)_{t\geq 0}$ under $\P_n$ converge as $n\to \infty$, in the sense of Skorohod's $J_1$-topology on the space of right-continuous non-decreasing functions, towards the law of the process $\eta_1=\left(\eta_1(t)\right)_{t\geq 0}$.
 \end{corollary} 

\begin{proof} In short, the claim follows from a general scaling limit theorem for integer valued Markov chains in \cite{BeKor} (see also \cite{HaMier} for an earlier work in this area), the limiting process being then a self-similar Markov process on $\R_+$. Theorem \ref{T1} is the key for checking that the framework of \cite{BeKor} indeed applies.

To start with, we observe from  the strong Markov property of $\left(n,S(n))\right)_{n\geq 0}$, that the conditional law of the elephant random walk started at the time
of  its $j$-th return to $0$,  $(S(n))_{n\geq \zeta(j)}$, given $\zeta(j)=k$,  is   the law $\P_k$ of the elephant started at time $k$ from $0$. It follows that $(\zeta_k(j))_{j\geq 0}$
is a homogeneous Markov chain with transition probabilities
given in the notation \eqref{E:defRk} by 
\begin{equation}\label{E:deftail}
\P_k(\zeta_k(1)=n+k) =\P_k(R=n) \qquad k,n\in 2\N. 
\end{equation}

We aim at applying \cite[Theorem 1]{BeKor}, and have to check the conditions denoted by (A1) and (A2) there.
The first follows directly from Theorem \ref{T1}. Indeed,  using  \eqref{E:estab} and \eqref{E:estab'}, we deduce that for any $t>0$, 
\begin{equation} \label{E:limtR}
\lim_{k\to \infty} k^{1/2} \P_k(R>tk) =  \sqrt{\frac{6-8p}{\pi (( t+1)^{3-4p}-1)}}.
\end{equation}
In the notation of \cite{BeKor}, this yields (A1) with\footnote{Beware that the sequence denoted by $(a_n)$ of \cite{BeKor}  is not the one defined by  \eqref{E:defa} here !}
$a_n=\sqrt n$ and $\Pi(\dd x)$ given by \eqref{E:Pi}. 

We next need to verify the condition (A2) of \cite{BeKor}. Since the L\'evy process $Y$ associated to $\eta_1$ by the transformation of  Lamperti is a subordinator with no drift, the latter reduces to checking that 
\begin{equation} \label{E:condB}
\lim_{k\to \infty} \sqrt k \E_k(1\wedge \log(1+R/k)) = \int_{(0,\infty)} (1\wedge x) \Pi(\dd x).
\end{equation}
In this direction, we write
$$ \E_k(1\wedge \log(1+R/k))=\int_0^{\e-1} \frac{1}{1+t} \P_k(R>kt) \dd t.$$
We readily infer from \eqref{E:domin}, Proposition \ref{P1} and Lemma \ref{L5} that 
$$\P_k(R>kt)\leq c/\sqrt{kt}, \qquad \text{for all } t\in(0,\e-1],$$
so that, by \eqref{E:limtR} and  dominated convergence,
\begin{align*}& \lim_{k\to \infty} \sqrt k \E_k(1\wedge \log(1+R/k))\\
 &=  \sqrt{\frac{6-8p}{\pi}}  \int_0^{\e-1} \frac{\dd t }{(1+t)\sqrt{ ( 1+t)^{3-4p}-1}} \\
&=  \sqrt{\frac{2(3-4p)}{\pi}}  \int_0^{1} \frac{\dd u }{\sqrt{ \e^{(3-4p)u}-1}}\\
&=  \sqrt{\frac{(3-4p)^3}{2\pi}}  \int_0^{\infty} (1\wedge x)  \left( \e^{(3-4p)s}-1\right)^{-3/2} \e^{(3-4p)x} \dd x.
\end{align*}
We have thus checked \eqref{E:condB} and the proof is complete. 
\end{proof}

  \begin{remark}\label{R:Fell2} The set of zeros $\mathcal Z$ of the elephant coincides with the range of the Markov chain $(\zeta(j))_{j\geq 0}$, and then
  \eqref{E:TZ} points at the convergence  in some appropriate distributional sense of $n^{-1}\mathcal Z$  to  the closed range of $\eta$,
  $\{\eta(t): t\geq 0\}^{\mathrm{(cl)}}$. On the one hand, Proposition \ref{P2} enables us to identify the latter with the closed range of $\lambda^{1/(3-4p)}$,
  $$\{\eta(t): t\geq 0\}^{\mathrm{(cl)}} = \{\lambda(s)^{1/(3-4p)}: s\geq 0\}^{\mathrm{(cl)}}.$$
 On the other hand, the closed range of $\lambda$ is precisely the zeros set  of the Brownian motion. This provides another rough argument for the weak convergence 
 \eqref{E:Fell}.
    \end{remark}
    
    \vskip 1cm
    \noindent\textbf{Acknowledgment.} I am grateful to two anonymous referees for their careful reading of the first version of this work and their constructive comments.

\bibliography{elephant.bib}

\begin{thebibliography}{10}

\bibitem{ACK}
{\sc Angel, O., Crawford, N., and Kozma, G.}
\newblock Localization for linearly edge reinforced random walks.
\newblock {\em Duke Math. J. 163}, 5 (2014), 889--921.

\bibitem{BaurBer}
{\sc Baur, E., and Bertoin, J.}
\newblock Elephant random walks and their connection to {P}\'olya-type urns.
\newblock {\em Phys. Rev. E 94\/} (Nov 2016), 052134.

\bibitem{Bercu}
{\sc Bercu, B.}
\newblock A martingale approach for the elephant random walk.
\newblock {\em J. Phys. A 51}, 1 (2018), 015201, 16.

\bibitem{BerLau1}
{\sc Bercu, B., and Laulin, L.}
\newblock On the multi-dimensional elephant random walk.
\newblock {\em J. Stat. Phys. 175}, 6 (2019), 1146--1163.

\bibitem{BerLau2}
{\sc Bercu, B., and Laulin, L.}
\newblock On the center of mass of the elephant random walk.
\newblock {\em Stochastic Processes and their Applications 133\/} (2021), 111
  -- 128.

\bibitem{Marco}
{\sc Bertengui, M.}
\newblock Functional limit theorems for the multi-dimensional elephant random
  walk.
\newblock arXiv:2004.02004.

\bibitem{BeCa}
{\sc {Bertoin}, J., and {Caballero}, M.-E.}
\newblock {Entrance from \(0+\) for increasing semi-stable Markov processes}.
\newblock {\em {Bernoulli} 8}, 2 (2002), 195--205.

\bibitem{BeKor}
{\sc {Bertoin}, J., and {Kortchemski}, I.}
\newblock {Self-similar scaling limits of Markov chains on the positive
  integers}.
\newblock {\em {Ann. Appl. Probab.} 26}, 4 (2016), 2556--2595.

\bibitem{ColGavSch1}
{\sc Coletti, C.~F., Gava, R., and Sch\"{u}tz, G.~M.}
\newblock Central limit theorem and related results for the elephant random
  walk.
\newblock {\em J. Math. Phys. 58}, 5 (2017), 053303, 8.

\bibitem{ColGavSch2}
{\sc Coletti, C.~F., Gava, R., and Sch\"{u}tz, G.~M.}
\newblock A strong invariance principle for the elephant random walk.
\newblock {\em J. Stat. Mech. Theory Exp.}, 12 (2017), 123207, 8.

\bibitem{ColPapa}
{\sc Coletti, C.~F., and Papageorgiou, I.}
\newblock Asymptotic analysis of the elephant random walk.
\newblock {\em Journal of Statistical Mechanics: Theory and Experiment 2021}, 1
  (jan 2021), 013205.

\bibitem{FHM}
{\sc Fan, X., Hu, H., and Ma, X.}
\newblock Cram{\'{e}}r moderate deviations for the elephant random walk.
\newblock {\em Journal of Statistical Mechanics: Theory and Experiment 2021}, 2
  (feb 2021), 023402.

\bibitem{Gut}
{\sc Gut, A., and Stadtm\"{u}ller, U.}
\newblock The number of zeros in {E}lephant random walks with delays.
\newblock {\em Statist. Probab. Lett. 174\/} (2021), 109112.

\bibitem{HaMier}
{\sc {Haas}, B., and {Miermont}, G.}
\newblock {Self-similar scaling limits of non-increasing Markov chains}.
\newblock {\em {Bernoulli} 17}, 4 (2011), 1217--1247.

\bibitem{Heyde}
{\sc Heyde, C.}
\newblock Asymptotics and criticality for a correlated {B}ernoulli process.
\newblock {\em Australian \& New Zealand Journal of Statistics 46}, 1 (2004),
  53--57.

\bibitem{Janson}
{\sc Janson, S.}
\newblock Functional limit theorems for multitype branching processes and
  generalized {P}\'olya urns.
\newblock {\em Stochastic Processes and their Applications 110}, 2 (2004), 177
  -- 245.

\bibitem{Kal}
{\sc Kallenberg, O.}
\newblock {\em Foundations of modern probability}, second~ed.
\newblock Probability and its Applications (New York). Springer-Verlag, New
  York, 2002.

\bibitem{KuTa}
{\sc Kubota, N., and Takei, M.}
\newblock Gaussian fluctuation for superdiffusive elephant random walks.
\newblock {\em J. Stat. Phys. 177}, 6 (2019), 1157--1171.

\bibitem{Kypri}
{\sc {Kyprianou}, A.~E.}
\newblock {\em {Fluctuations of L\'evy processes with applications.
  Introductory lectures. 2nd ed}}, 2nd ed.~ed.
\newblock Berlin: Springer, 2014.

\bibitem{Lamperti}
{\sc {Lamperti}, J.}
\newblock {Semi-stable Markov processes. I}.
\newblock {\em {Z. Wahrscheinlichkeitstheor. Verw. Geb.} 22\/} (1972),
  205--225.

\bibitem{Latala}
{\sc {Lata{\l}a}, R.}
\newblock {Estimation of moments of sums of independent real random variables}.
\newblock {\em {Ann. Probab.} 25}, 3 (1997), 1502--1513.

\bibitem{Mahmoud}
{\sc Mahmoud, H.~M.}
\newblock {\em P\'{o}lya urn models}.
\newblock Texts in Statistical Science Series. CRC Press, Boca Raton, FL, 2009.

\bibitem{MerRol}
{\sc Merkl, F., and Rolles, S. W.~W.}
\newblock Recurrence of edge-reinforced random walk on a two-dimensional graph.
\newblock {\em Ann. Probab. 37}, 5 (2009), 1679--1714.

\bibitem{Meyer}
{\sc Meyer, P.~A.}
\newblock D\'{e}monstration simplifi\'{e}e d'un th\'{e}or\`eme de {K}night.
\newblock In {\em S\'{e}minaire de {P}robabilit\'{e}s, {V} ({U}niv.
  {S}trasbourg, ann\'{e}e universitaire 1969--1970)}. 1971, pp.~191--195.
  Lecture Notes in Math., Vol. 191.

\bibitem{Molchanov}
{\sc {Molchanov}, I.}
\newblock {\em {Theory of random sets. 2nd edition}}, 2nd edition~ed., vol.~87.
\newblock London: Springer, 2017.

\bibitem{Papangelou}
{\sc Papangelou, F.}
\newblock Integrability of expected increments of point processes and a related
  random change of scale.
\newblock {\em Trans. Amer. Math. Soc. 165\/} (1972), 483--506.

\bibitem{Pem}
{\sc Pemantle, R.}
\newblock A survey of random processes with reinforcement.
\newblock {\em Probab. Surveys 4\/} (2007), 1--79.

\bibitem{RY}
{\sc Revuz, D., and Yor, M.}
\newblock {\em Continuous martingales and {B}rownian motion}, third~ed.,
  vol.~293 of {\em Grundlehren der Mathematischen Wissenschaften [Fundamental
  Principles of Mathematical Sciences]}.
\newblock Springer-Verlag, Berlin, 1999.

\bibitem{WillRo}
{\sc Rogers, L. C.~G., and Williams, D.}
\newblock {\em Diffusions, {M}arkov processes, and martingales. {V}ol. 1},
  second~ed.
\newblock Wiley Series in Probability and Mathematical Statistics: Probability
  and Mathematical Statistics. John Wiley \& Sons, Ltd., Chichester, 1994.
\newblock Foundations.

\bibitem{Sabot}
{\sc Sabot, C.}
\newblock Random walks in random {D}irichlet environment are transient in
  dimension {$d\geq 3$}.
\newblock {\em Probab. Theory Related Fields 151}, 1-2 (2011), 297--317.

\bibitem{SabZen}
{\sc Sabot, C., and Zeng, X.}
\newblock A random {S}chr\"{o}dinger operator associated with the vertex
  reinforced jump process on infinite graphs.
\newblock {\em J. Amer. Math. Soc. 32}, 2 (2019), 311--349.

\bibitem{SchTr}
{\sc Sch\"utz, G.~M., and Trimper, S.}
\newblock Elephants can always remember: Exact long-range memory effects in a
  non-{M}arkovian random walk.
\newblock {\em Phys. Rev. E 70\/} (Oct 2004), 045101.

\bibitem{Williams}
{\sc {Williams}, D.}
\newblock {L\'evy's downcrossing theorem}.
\newblock {\em {Z. Wahrscheinlichkeitstheor. Verw. Geb.} 40\/} (1977),
  157--158.

\end{thebibliography}

\end{document}